\documentclass[12pt]{article}%
\usepackage{amsmath}
\usepackage{amsfonts}
\usepackage{xcolor}
\usepackage{soul}
\usepackage{amssymb}
\usepackage{graphicx}%
\providecommand{\U}[1]{\protect\rule{.1in}{.1in}}
\setstcolor{red}
\newtheorem{theorem}{Theorem}

\newtheorem{corollary}[theorem]{Corollary}

\newtheorem{definition}[theorem]{Definition}

\newtheorem{proposition}[theorem]{Proposition}
\newtheorem{remark}[theorem]{Remark}

\newenvironment{proof}[1][Proof]{\noindent\textbf{#1.} }{\ \rule{0.5em}{0.5em}}
\begin{document}

\title{Molecular chains interacting by Lennard-Jones and Coulomb forces}
\author{Carlos Garc\'{\i}a-Azpeitia\thanks{E-mail: cgazpe@ciencias.unam.mx}, Manuel
Tejada-Wriedt\thanks{E-mail: mtw@ciencias.unam.mx}\\{\small Departamento de Matem\'{a}ticas, Facultad de Ciencias, }\\{\small Universidad Nacional Aut\'{o}noma de M\'{e}xico, Ciudad de M\'{e}xico,
04510, M\'{e}xico}}
\maketitle

\begin{abstract}
We study equations for the mechanical movement of chains of identical
particles in the plane interacting with their nearest-neighbors by bond
stretching and by van der Waals and Coulomb forces. We find collinear and
circular equilibria as minimizers of the energy potential for chains with
Neumann and periodic boundary conditions. We prove global bifurcation of
periodic brake orbits from these equilibria applying the global Rabinowitz
alternative. These results are complemented with numeric computations for
ranges of parameters that include carbon atoms among other molecules.


MSC 34C25, 37G40, 47H11, 70H33

\end{abstract}

\section*{Introduction}

Molecular mechanics have been very successful in describing both small
molecules and large biological systems. They are built on the framework of
classical mechanics, and rely on the accurate description of atomic
interactions. The potential energy of all systems in molecular mechanics is
represented by what in chemistry is known as a force field, which refers to
the functional form of the potential energy and the set of parameters (e.g.
bond strength, electric charge, van der Waals radius, etc.) that describe how
the particles of a given system interact. Numerous force fields have been
developed, from \textquotedblleft all-atom" to \textquotedblleft
coarse-grained". The first ones take into account every atom in a system, and
the second ones treat groups of atoms as single particles (see \cite{LoGu} and
\cite{MoTi} for review and state of the art).

Most classical force fields used in molecular mechanics have potential energy
terms associated with bond deformations, electrostatic interactions, and van
der Waals forces. We consider a chain of identical particles in the plane
$u_{j}\in\mathbb{R}^{2}$ for $j=1,...,n$, where each particle $u_{j}$
interacts with its nearest neighbors $u_{j-1}$ and $u_{j+1}$ by bond
stretching (no bending forces are considered), and with the rest of the chain
by van der Waals and electrostatic forces, modeled with Lennard-Jones and
Coulomb potentials.

The adimensionalized equations for the system of $n$ particles are%
\begin{equation}
\ddot{u}_{j}=-V_{u_{j}},\text{\qquad}j=1,\ldots,n\text{,} \label{Eq}%
\end{equation}
where the energy function is%
\begin{equation}
V=\sum_{j=1}^{n-1}U(\left\vert u_{j+1}-u_{j}\right\vert ^{2})+\sigma
U(\left\vert u_{n}-u_{1}\right\vert ^{2})+\sum_{1\leq j<k\leq n}W(\left\vert
u_{j}-u_{k}\right\vert ^{2}). \label{V}%
\end{equation}
Bond stretching is represented by the potential%
\begin{equation}
U(x)=x-2x^{1/2}\text{,} \label{U}%
\end{equation}
while non-bonded interactions by a potential $W\in C^{2}$ such that
\begin{equation}
\lim_{x\rightarrow0}W(x)=\lim_{x\rightarrow0}-W^{\prime}(x)=\infty,\qquad
\lim_{x\rightarrow\infty}W(x)=\lim_{x\rightarrow\infty}W^{\prime}(x)=0\text{.}
\label{W}%
\end{equation}
The assumptions for $W$ assure that the interaction is repulsive when two
particles are close and vanish when they are far from each other, which is the
case of van der Waals and electrostatic interactions.

Two kind of molecular chains are studied: the circular chain (periodic
boundary condition), where $\sigma=1$ in (\ref{V}), and the collinear chain
(Neumann boundary conditions), where $\sigma=0$ in (\ref{V}). The equations
for the circular chain are equivariant under the symmetry group
\[
D_{n}\times O(2)\times O(2)\text{,}%
\]
which acts by permuting particles, rotating positions and translating time;
see (\ref{Ac}) and (\ref{AcT}) for details. For the collinear chain, the
equations are equivariant only under the group $\mathbb{Z}_{2}\times
O(2)\times O(2)$.

In Theorem \ref{The1}, we use the Palais principle of symmetric criticality to
obtain equilibria as minimizers of $V$ in subspaces of symmetric
configurations.\ This allows us to prove the existence of symmetric collinear
and circular equilibria, among others.

Theorems \ref{The2} and \ref{The3} establish that both the collinear and
circular equilibria have a global bifurcation of $2\pi/\nu$-periodic solutions
emanating from the frequency $\nu=\nu_{0}$ for each positive non-resonant
eigenvalue $\nu_{0}^{2}$ of the Hessian of $V$. The global property is proved
using the global Rabinowitz alternative in subspaces of symmetric periodic
functions. This property assures us that the branch is a continuum. Moreover,
the branch has norm or period going to infinity, ends in a collision, or comes
back to other bifurcation point.

The solutions given in Theorems \ref{The2} and \ref{The3} are brake orbits
(Figure 1). These kind of orbits are solutions for which all the velocities
are zero at some instant, see \cite{Bar93}, \cite{MoMo} and references
therein. Bifurcation of other types of periodic solutions for molecules have
been considered previously in \cite{MoRo99}.

\begin{figure}[h]
\begin{center}
\resizebox{13cm}{!} {\includegraphics{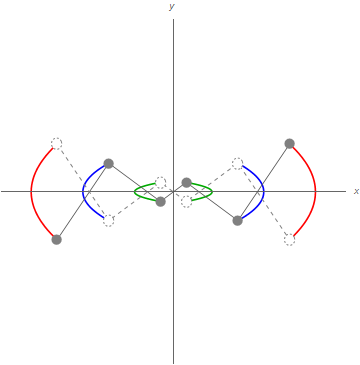}  \hskip1.0cm
\includegraphics{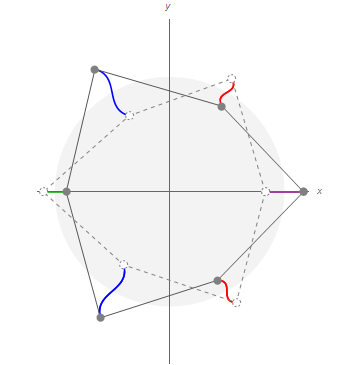} }
\end{center}
\caption{Illustration of the symmetries of brake orbits for $n=6$. Orbits with
the same color are symmetric by reflections and rotations.}%
\end{figure}

We complement our results with numeric computations for parameters that
include carbon atoms, among other particles considered in CHARMM36 force field
\cite{BeZh12}. In Figures 2 and 3, we present the amplitude of the collinear
and circular equilibria for $n=6$, respectively, and the number of negative
eigenvalues of the Hessian. The numeric computations allow us to conclude that
both equilibria for $6$ general particles loose stability when the
Lennard-Jones parameter $A$ is increased. In Table 1, we present the number of
negative eigenvalues of the collinear and circular equilibria for different
number of carbon atoms.

In Section 1, we prove the existence of minimizers that correspond to
equilibria. In Section 2, we find the linearization around them. In Section 3,
we prove global bifurcation of periodic solutions. In Section 4, we
numerically estimate the amplitude and spectra of the equilibria for $n=6$,
and we discuss their stability.

\section{Equilibria of molecular chains}

Equilibrium configurations of molecular chains correspond to critical points
of $V$; these are points $\mathbf{a}=(a_{1},...,a_{n})\in\mathbb{R}^{2n}$ such
that $\nabla_{u}V(\mathbf{a})=0$.

The potential $V$ is well defined in the set $\Omega=\{u\in\mathbb{R}%
^{2n}:u_{j}\neq u_{k}\}$. Given that any translation of an equilibrium is an
equilibrium, we restrict the potential $V$ to the subspace
\[
\Omega_{0}=\{u\in\Omega:\sum_{j=1}^{n}u_{j}=0\}\text{.}%
\]

We define the action of the group $S_{n}$ of permutations of $\{1,...,n\}$ and
the group $O(2)=S^{1}\cup\tilde{\kappa}S^{1}$in $\mathbb{R}^{2n}$ as%
\begin{equation}
\rho(\gamma)x_{j}=x_{\gamma(j)}\text{,\qquad}\rho(\theta)x_{j}=e^{-J\theta
}x_{j}\text{,\qquad}\rho(\tilde{\kappa})x_{j}=Rx_{j}\text{,} \label{Ac}%
\end{equation}
where the matrices $J$ and $R$ are%
\[
J=\left(
\begin{array}
[c]{cc}%
0 & -1\\
1 & 0
\end{array}
\right)  \text{ and }R=\left(
\begin{array}
[c]{cc}%
1 & 0\\
0 & -1
\end{array}
\right)  \text{.}%
\]

Let us assume for the moment that $V$ is $G$-invariant, where $G$ is a
subgroup of $S_{n}\times O(2)$.

\begin{theorem}
\label{The1}If $H$ is a subgroup of $G$, then the minimizer of $V$ is achieved
in each connected component of%
\[
\Omega_{0}^{H}=\Omega_{0}\cap\mathrm{Fix}(H)=\{u\in\Omega_{0}:\rho(g)u=u\text{
for }g\in H\}\text{.}%
\]
These minimizers are critical points of $V$ in $\Omega$.
\end{theorem}

\begin{proof}
The restricted potential $V:\Omega_{0}\subset\mathbb{R}^{2(n-1)}%
\rightarrow\mathbb{R}$ satisfies $V(u)\rightarrow+\infty$ when $u_{j}%
\rightarrow u_{j+1}$, since the non-bonded interactions are such that
$\lim_{x\rightarrow0}W(x)=\infty$. Also, given that $\lim_{x\rightarrow\infty
}U(x)=\infty$, then $V(u)\rightarrow+\infty$ when $u\rightarrow\infty$.
Therefore, the potential $V$ is coercive and goes to infinity in the boundary
of $\Omega_{0}$. We conclude that $V$ has a minimizer in each connected
component of $\Omega_{0}^{H}$.

By the Palais principle of symmetric criticality \cite{Pa79}, each minimizer
is a critical point of $V$ in $\Omega_{0}$. Moreover, a critical point of $V$
in $\Omega_{0}$ satisfies $\nabla V(\mathbf{a})+\lambda_{1}\mathbf{e}%
_{1}+\lambda_{2}\mathbf{e}_{2}=0$ with $\mathbf{a}\in\Omega_{0}$. The
invariance of $V$ under translations implies that $\nabla V(\mathbf{a}%
)\cdot\mathbf{e}_{j}=0$ for $j=1,2$, so $\lambda_{j}=0$. Consequently, a
minimizer of $V$ in a connected component of $\Omega_{0}^{H}$ is a critical
point of $V$ in $\Omega$, i.e. $\nabla V(\mathbf{a})=0$.
\end{proof}

We have used the property of stratification of space to guaranty the existence
of a different equilibrium for each maximal isotropy group $H$ of $G$.

\subsection{The collinear chain}

In the case of Neumann conditions ($\sigma=0$), the potential $V$ is invariant
under the group
\[
G=\mathbb{Z}_{2}(\kappa)\times O(2)\text{,}%
\]
where $\mathbb{Z}_{2}(\kappa)$ is the group generated by the permutation given
by%
\[
\kappa(j)=n+1-j\text{.}%
\]

Therefore, the minimizer of $V$ is achieved in each connected component of
$\Omega_{0}^{H}$ with the maximal subgroup%
\[
H=\mathbb{Z}_{2}(\kappa,\pi)\times\mathbb{Z}_{2}(\tilde{\kappa})\text{.}%
\]
Since $(\kappa,\pi)$ and $\tilde{\kappa}$ generate $H$, the minimizer
$\mathbf{a}$ in the fixed point space have components $a_{j}\in\mathbb{R}$ and
$a_{n+1-j}=-a_{j}$.

In fact, the domain $\Omega_{0}^{H}$ has many connected components. Since, in
the chain, $u_{j}$ is coupled only to the adjacent $u_{j-1}$ and $u_{j+1}$, we
only consider the component
\[
\left\{  u\in\Omega_{0}\cap\mathrm{Fix}(H):u_{j}<u_{j+1}\right\}  \text{,}%
\]
which has physical meaning.

\begin{corollary}
There is a symmetric collinear equilibrium $\mathbf{a}=(a_{1},...,a_{n})$ such
that $a_{j}\in\mathbb{R}$ satisfy%
\[
a_{1}<a_{2}<...<a_{n}\text{,}%
\]
where $a_{1}=-a_{n}$, $a_{2}=-a_{n-1}$, etc.
\end{corollary}

Since $\mathbf{a}$ is a minimizer in $\Omega_{0}^{H}$, whose dimension is
bigger than $(n-1)/2$, then the Hessian $D^{2}V(\mathbf{a)}$ has at least
$(n-1)/2$ positive eigenvalues.

\subsection{The circular chain}

For periodic boundary conditions ($\sigma=1$), the potential $V$ is invariant
under the group
\[
G=D_{n}\times O(2)\text{,}%
\]
where $D_{n}$ is the group generated by the permutations%
\[
\zeta(j)=j+1,\qquad\kappa(j)=n-j\text{,}%
\]
modulus $n$.

Then, there is a minimizer with isotropy group%
\[
\tilde{D}_{n}=\left\langle (\zeta,\zeta),(\kappa,\tilde{\kappa})\right\rangle
<G\text{,}%
\]
where $\zeta=2\pi/n$.

\begin{definition}
Let
\[
s_{k}=2\sin\frac{k\zeta}{2}=2\sin\frac{k\pi}{n}\text{.}%
\]
Then, $s_{k}^{2}=s_{-k}^{2}$ and%
\[
\left\vert 1-e^{ik\zeta}\right\vert ^{2}=2(1-\cos k\zeta)=4\sin^{2}%
k\zeta/2=s_{k}^{2}\text{.}%
\]

\end{definition}

We describe the circular equilibrium explicitly in the following proposition.

\begin{proposition}
We define $a_{j}=ae^{ij\zeta}$ for $j=1,...,n$ with $\zeta=2\pi/n$, where we
have identified the complex plane with the real plane. Let $a$ be such that
\[
U^{\prime}\left(  a^{2}s_{1}^{2}\right)  =-\frac{1}{2s_{1}^{2}}\sum
_{k=1}^{n-1}W^{\prime}\left(  a^{2}s_{k}^{2}\right)  s_{k}^{2}\text{,}%
\]
then $\mathbf{a}=(a_{1},...,a_{n})$ is a critical point of $V$.
\end{proposition}

\begin{proof}
We have that%
\[
a_{j}-a_{k}=a_{j}(1-e^{i(k-j)\zeta})\text{ and }\left\vert a_{j}%
-a_{k}\right\vert =as_{k-j}\text{.}%
\]
Then
\begin{align*}
\sum_{k=j\pm1}U^{\prime}(\left\vert a_{j}-a_{k}\right\vert ^{2})2\left(
a_{j}-a_{k}\right)   &  =a_{j}U^{\prime}\left(  a^{2}s_{1}^{2}\right)
\sum_{k=j\pm1}2(1-e^{i(k-j)\zeta})\\
&  =a_{j}U^{\prime}\left(  a^{2}s_{1}^{2}\right)  \left(  2s_{1}^{2}\right)
\text{,}%
\end{align*}
and%
\begin{align*}
\sum_{k\neq j}W^{\prime}(\left\vert a_{j}-a_{k}\right\vert ^{2})2\left(
a_{j}-a_{k}\right)   &  =a_{j}\sum_{k=1(k\neq j)}^{n}W^{\prime}\left(
a^{2}s_{k-j}^{2}\right)  2(1-e^{i(k-j)\zeta})\\
&  =a_{j}\left(  \sum_{k=1}^{n-1}W^{\prime}\left(  a^{2}s_{k}^{2}\right)
s_{k}^{2}\right)  \text{.}%
\end{align*}
Therefore, the derivative of $V$ at $\mathbf{a}$ is
\[
V_{u_{j}}(\mathbf{a})=a_{j}\left(  2U^{\prime}\left(  a^{2}s_{1}^{2}\right)
s_{1}^{2}+\sum_{k=1}^{n-1}W^{\prime}\left(  a^{2}s_{k}^{2}\right)  s_{k}%
^{2}\right)  \text{.}%
\]
The proposition follows from observing that the term in parenthesis is
independent of $j$.
\end{proof}

The previous proposition is independent of the particular form of the
potentials $U$ and $W$. For the bonding potential we have $U^{\prime
}(x)=1-x^{-1/2}$.

\begin{corollary}
Let
\[
S(a)=\frac{1}{as_{1}}-\frac{1}{2s_{1}^{2}}\sum_{k=1}^{n-1}W^{\prime}\left(
a^{2}s_{k}^{2}\right)  s_{k}^{2}\text{.}%
\]
Using the properties of $W$, we have that $S(a)\rightarrow\infty$ as
$a\rightarrow0$ and $S(a)\rightarrow0$ as $a\rightarrow+\infty$. Then, there
is at least one solution of $S(\alpha)=1$, and $\mathbf{a}=(a_{1},...,a_{n})$
is a circular equilibrium, where $a_{j}=\alpha e^{ij\zeta}$.
\end{corollary}

The polygon $\mathbf{a}$ is a local minimizer in $\Omega_{0}^{H}$. Therefore,
the matrix $D^{2}V(\mathbf{a)}$ has at least $1$ positive eigenvalue.

\begin{remark}
Depending on the form of $W$ there may exist other circular equilibria. Other
critical points different to the circular chain exist in this case; for
example, there is a symmetric collinear equilibrium in the fixed point space
of $H=\mathbb{Z}_{2}(\kappa,\pi)\times\mathbb{Z}_{2}(\tilde{\kappa})$.
\end{remark}

\section{Linearization}

We define $A_{ij}$ as the $2\times2$ minors of the Hessian $D^{2}%
V(\mathbf{a})\in M_{\mathbb{R}}(2n)$, i.e.
\[
D^{2}V(\mathbf{a})=(A_{ij})_{ij=1}^{n}\text{.}%
\]
Due to the form of the potential $V$, the minors satisfy%
\[
A_{ii}=-\sum_{j\neq i}A_{ij}\text{.}%
\]

\subsection{The collinear chain}

When $\mathbf{a}$ corresponds to the collinear equilibrium, the Hessian
$D^{2}V(\mathbf{a)}$ is a linear $\mathbb{Z}_{2}(\tilde{\kappa})$-equivariant
map. This implies that $A_{ij}R=RA_{ij}$. Then, $A_{ij}$ is diagonal and
$D^{2}V(\mathbf{a)}$ has an equivalent form as $\mathrm{diag}(M_{0},M_{1})$,
where the blocks $M_{k}$ correspond to different representations of
$\mathbb{Z}_{2}(\tilde{\kappa})$.

Let $\delta_{ij}=1$ if $\left\vert i-j\right\vert =1$ and $\delta_{ij}=0$
otherwise. Then
\[
V=\sum_{1\leq j<k\leq n}[\delta_{jk}U+W](\left\vert u_{j}-u_{k}\right\vert
^{2})\text{,}%
\]
where%
\[
\lbrack\delta_{jk}U+W](\left\vert u_{j}-u_{k}\right\vert ^{2})=\delta
_{jk}U(\left\vert u_{j}-u_{k}\right\vert ^{2})+W(\left\vert u_{j}%
-u_{k}\right\vert ^{2})\text{.}%
\]
Therefore,
\[
\nabla_{u_{j}}V(u)=\sum_{k\neq j}2\left(  u_{j}-u_{k}\right)  [\delta
_{kj}U+W]^{\prime}(\left\vert u_{j}-u_{k}\right\vert ^{2})\text{.}%
\]

If $i\neq j$, then
\begin{align*}
-A_{ij}  &  =-D_{u_{i}}\nabla_{u_{j}}V=-D_{u_{i}}2\left(  u_{j}-u_{i}\right)
[\delta_{ij}U+W]^{\prime}(\left\vert u_{j}-u_{i}\right\vert ^{2})\\
&  =2[\delta_{ij}U+W]^{\prime}\left(  (a_{j}-a_{i})^{2}\right)  I+4\left(
a_{j}-a_{i}\right)  ^{2}[\delta_{ij}U+W]^{\prime\prime}(\left(  a_{j}%
-a_{i}\right)  ^{2})\mathrm{diag}(1,0)\text{.}%
\end{align*}
Explicitly, we have the following equivalence of matrices.

\begin{proposition}
The matrix $D^{2}V(\mathbf{a)}$ is equivalent to
\[
\mathrm{diag}(M_{0},M_{1}),
\]
where $M_{0}=(a_{ij})_{i,j}^{n}$, $M_{1}=(b_{ij})_{i,j=1}^{n}$, and
\[
-a_{ij}=2[\delta_{ij}U+W]^{\prime}(\left(  a_{j}-a_{i}\right)  ^{2})+4\left(
a_{j}-a_{i}\right)  ^{2}[\delta_{ij}U+W]^{\prime\prime}(\left(  a_{j}%
-a_{i}\right)  ^{2})
\]
and
\[
-b_{ij}=2[\delta_{ij}U+W]^{\prime}(\left(  a_{j}-a_{i}\right)  ^{2})
\]
for $i\neq j$. Moreover, $a_{ii}=\sum_{j\neq i}-a_{ij}$ and $b_{ii}%
=\sum_{j\neq i}-b_{ij}$.
\end{proposition}

\subsection{The circular chain}

The circular chain in real coordinates is given by $\mathbf{a}=(a_{1}%
,...,a_{n})$, where $a_{j}=ae^{Jj\zeta}e_{1}$ and $e_{1}=(1,0)\in
\mathbb{R}^{2}$.

\begin{proposition}
Using the transform%
\begin{equation}
T_{k}(z)=(n^{-1/2}e^{(ikI+J)\zeta}z,...,n^{-1/2}e^{n(ikI+J)\zeta}%
z):\mathbb{C}^{2}\rightarrow\mathbb{C}^{2n}\text{,} \label{T}%
\end{equation}
the matrix $D^{2}V(\mathbf{a})$, as a matrix in $M_{\mathbb{C}}(2n\mathbb{)}$,
satisfies that
\[
D^{2}V(\mathbf{a})T_{k}(z)=T_{k}(M_{k}z),
\]
i.e. the Hessian decomposes in blocks%
\[
M_{k}=\sum_{j=1}^{n-1}(-A_{nj})(I-e^{j(ikI+J)\zeta})\text{ for }%
k\in\{1,...,n\}\text{.}%
\]

\end{proposition}

\begin{proof}
We use that $D^{2}V(\mathbf{a})\in M_{\mathbb{R}}(2n\mathbb{)}$ is
$\mathbb{Z}_{n}(\zeta,\zeta)$-equivariant. See Proposition 7 of \cite{GaIz11}.
\end{proof}

This formula has been used to study the stability of the polygonal equilibrium
for bodies and vortices in \cite{GaIz13a,GaIz13b}. In the next proposition, we
calculate $M_{k}$ explicitly.

\begin{proposition}
Let $\delta_{j}=1$ if $j=1,n-1$ and $\delta_{j}=0$ otherwise, and%
\[
b_{j}=2[\delta_{j}U+W]^{\prime}\left(  a^{2}s_{j}^{2}\right)  ,\qquad
c_{j}=2s_{j}^{2}[\delta_{j}U+W]^{\prime\prime}\left(  a^{2}s_{j}^{2}\right)
\text{.}%
\]
Then,
\[
M_{k}=\alpha_{k}I+\beta_{k}R-\gamma_{k}(iJ)\text{,}%
\]
where%
\begin{align*}
\alpha_{k}  &  =\sum_{j=1}^{n-1}\left(  b_{j}+c_{j}\right)  \left(  1-\cos
kj\zeta\cos j\zeta\right)  \text{,}\\
\gamma_{k}  &  =\sum_{j=1}^{n-1}\left(  b_{j}+c_{j}\right)  \sin jk\zeta\sin
j\zeta\text{,}\\
\beta_{k}  &  =\sum_{j=1}^{n-1}c_{j}\left(  \cos jk\zeta-\cos j\zeta\right)
\text{.}%
\end{align*}

\end{proposition}

\begin{proof}
First we need to calculate $A_{nj}=D_{u_{n}}V_{u_{j}}(\mathbf{a})$. In real
coordinates, $a_{j}=ae^{Jj\zeta}e_{1}$. Then,
\[
\left(  a_{n}-a_{j}\right)  ^{T}\left(  a_{n}-a_{j}\right)  =a^{2}\left(
I-e^{Jj\zeta}\right)  ^{T}e_{1}^{T}e_{1}\left(  I-e^{Jj\zeta}\right)
=a^{2}C_{j}\text{,}%
\]
where $C_{j}$ is the matrix
\[
C_{j}=\left(  I-e^{Jj\zeta}\right)  ^{T}\mathrm{diag}(1,0)\left(
I-e^{Jj\zeta}\right)  \text{.}%
\]

For $j \neq n$,
\[
\nabla_{u_{n}}V(u)=\sum_{k=1}^{n-1}[\delta_{k}U+W]^{\prime}(\left\vert
u_{n}-u_{k}\right\vert ^{2})2\left(  u_{n}-u_{k}\right)  .
\]
Evaluating this expression at $\mathbf{a}$,
\[
-A_{nj} = 2[\delta_{j}U+W]^{\prime}(a^{2}s_{j}^{2})I+4a^{2}[\delta
_{j}U+W]^{\prime\prime}\left(  a^{2}s_{j}^{2}\right)  C_{j}\text{.}
\qquad(j\neq n)
\]
Finally,
\[
-A_{nj}=b_{j}I+\left(  2c_{j}/s_{j}^{2}\right)  C_{j}. \qquad(j\neq n)
\]

The matrix $C_{j}$ can be written explicitly as%
\[
C_{j}=\left(
\begin{array}
[c]{cc}%
(1-\cos j\zeta)^{2} & -(1-\cos j\zeta)\sin j\zeta\\
-(1-\cos j\zeta)\sin j\zeta & (\sin j\zeta)^{2}%
\end{array}
\right)  \text{.}%
\]
By means of the relation $\sin^{2}j\zeta=(1-\cos j\zeta)(1+\cos j\zeta)$,%
\[
C_{j}=(1-\cos j\zeta)\left(
\begin{array}
[c]{cc}%
(1-\cos j\zeta) & -\sin j\zeta\\
-\sin j\zeta & (1+\cos j\zeta)
\end{array}
\right)  =\frac{s_{j}^{2}}{2}(I-e^{jJ\zeta}R)\text{.}%
\]
Therefore, $-A_{nj}=\left(  b_{j}+c_{j}\right)  I-c_{j}e^{jJ\zeta}R$ and
\[
B_{k}=\sum_{j=1}^{n-1}\left(  b_{j}+c_{j}\right)  (I-e^{j(ikI+J)\zeta}%
)+c_{j}(e^{jikI}-e^{jJ\zeta})R.
\]
Since $b_{j}$ and $c_{j}$ satisfy $b_{n-j}=b_{j}$, $c_{n-j}=c_{j}$, using the
equalities%
\begin{align*}
e^{-(jJ\zeta)}+e^{(jJ\zeta)}  &  =2I\cos j\zeta\text{,}\\
e^{j(ikI+J)\zeta}+e^{-j(ikI+J)\zeta}  &  =2I\cos jk\zeta\cos j\zeta+2iJ\sin
jk\zeta\sin j\zeta,
\end{align*}
we can conclude that $M_{k}=\alpha_{k}I+\beta_{k}R-\gamma_{k}(iJ)$.
\end{proof}

The Hessian $D^{2}V(\mathbf{a})$ is $\mathbb{Z}_{2}(\kappa,\tilde{\kappa}%
)$-equivariant, so $M_{n-k}R=RM_{k}$. This means that blocks $M_{k}$ and
$M_{n-k}$ have the same eigenvalues $\lambda_{k}^{\pm}=\alpha_{k}\pm
\sqrt{\beta_{k}^{2}+\gamma_{k}^{2}}$. We can choose the corresponding
eigenvectors $v_{k}^{\pm}$ such that
\begin{equation}
Rv_{k}^{\pm}=-v_{n-k}^{\pm}\text{ and }\bar{v}_{k}^{\pm}=v_{n-k}^{\pm}.
\label{eig}%
\end{equation}

Actually, the real matrix $D^{2}V(\mathbf{a})\in M_{\mathbb{R}}(2n\mathbb{)}$
is equivalent to the matrix%
\[
\mathrm{diag}(M_{1},...,M_{n/2},M_{n})\text{, }%
\]
where $M_{k}\in M_{\mathbb{C}}(2\mathbb{)}$ for $k\in\lbrack1,n/2)\cap
\mathbb{N}$ and $M_{n/2},M_{n}\in M_{\mathbb{R}}(2\mathbb{)}$.

\section{Bifurcation of periodic solutions}

The bifurcation of periodic solutions corresponds to zeros of%
\[
f(x;\nu)=-\ddot{x}-\nu^{-2}\nabla V(\mathbf{a}+x)\text{,}%
\]
where $u(t)=\mathbf{a}+x(\nu t)$ and $\Omega=\{x\in\mathbb{R}^{2n}:x_{j}\neq
x_{i}\}$. To manage the translational symmetries, we define the restriction of
$f$ to the subspace
\[
X=\{x\in L_{2\pi}^{2}(\Omega):\int_{0}^{2\pi}\sum_{j=1}^{n}x_{j}dt=0\}\text{.}%
\]
The operator $f(x):H_{2\pi}^{2}\cap X\rightarrow X$ is well define because,
for $x\in X$,%
\[
\int_{0}^{2\pi}\sum_{j=1}^{n}f_{j}=\int_{0}^{2\pi}\sum_{j=1}^{n}V_{u_{j}}=0
\]
and $f\in X$.

For an equilibrium $\mathbf{a}$, $\nabla V(\mathbf{a})=0$. So $f(\mathbf{0}%
;\nu)=0$ for any $\nu$. Therefore,%
\[
f(x;\nu)=-\ddot{x}-\nu^{-2}D^{2}V(\mathbf{a})x+g(x)\text{,}%
\]
where $g(x)=O(\left\vert x\right\vert ^{2})$.

Let $K:L_{2\pi}^{2}\rightarrow H_{2\pi}^{2}$ be given in the Fourier basis
$x=\sum_{l\in\mathbb{Z}}x_{l}e^{ilt}\in L_{2\pi}^{2}$ as%
\[
Kx=x_{0}+\sum_{l\in\mathbb{Z}\backslash\{0\}}l^{-2}x_{l}e^{ilt}\text{.}%
\]
Since $K:H_{2\pi}^{2}\rightarrow H_{2\pi}^{2}$ is compact, the operator%
\[
Kf(x,\nu)=x-T(\nu)x+g(x):H_{2\pi}^{2}(\Omega)\times\mathbb{R}^{+}\rightarrow
H_{2\pi}^{2}(\Omega)
\]
is well defined, where
\[
T(\nu)x=\left(  \nu^{-2}D^{2}V(\mathbf{a})+I\right)  x_{0}+\sum_{l\in
\mathbb{Z}\backslash\{0\}}(l\nu)^{-2}D^{2}V(\mathbf{a})x_{l}e^{ilt}%
\]
is a linear \emph{compact map} and $g(x)=\mathcal{O}(\left\vert x\right\vert
_{H_{2\pi}^{2}}^{2})$ is a \emph{nonlinear compact} map.

We define the action of the group
\[
D_{n}\times O(2)\times O(2)\text{,}%
\]
in $X$ by $\rho(\gamma)x(t)$ for $\gamma\in D_{n}\times O(2)$ given in
(\ref{Ac}) and%
\begin{equation}
\rho(\varphi)x(t)=x(t+\varphi)\text{,\qquad}\rho(\bar{\kappa}%
)x(t)=x(-t)\text{.} \label{AcT}%
\end{equation}

\begin{definition}
We say that $\nu_{0}^{2}$ is a \emph{non-resonant} eigenvalue of
$D^{2}V(\mathbf{a})$ if $l^{2}\nu_{0}^{2}$ is not an eigenvalue of
$D^{2}V(\mathbf{a})$ for any integer $l\geq2$.
\end{definition}

\subsection{The collinear chain}

The map $Kf$ is equivariant under the action of
\[
\Gamma=\mathbb{Z}_{2}\times O(2)\times O(2)\text{,}%
\]
where $\mathbb{Z}_{2}$ is generated by the permutation $\kappa(j)=n+1-j$. We
used that the isotropy group of the collinear equilibrium $\mathbf{a}$ is%
\[
\Gamma_{\mathbf{a}}=\mathbb{Z}_{2}(\kappa,\pi)\times\mathbb{Z}_{2}%
(\tilde{\kappa})\times O(2)
\]
to show that $D^{2}V(\mathbf{a)}$ is equivalent to $\mathrm{diag}(M_{0}%
,M_{1})$, where $M_{0}$ has a zero-eigenvalue corresponding to translations
and $M_{1}$ has two zero-eigenvalues corresponding to translations and rotations.

\begin{theorem}
\label{The2}Assume that $D^{2}V(\mathbf{a})$ has only three zero eigenvalues,
two corresponding to translations and one to rotations. If $\nu_{k}^{2}$ is a
simple non-resonant positive eigenvalue of $D^{2}V(\mathbf{a})$ corresponding
to the block $M_{k}\in M_{\mathbb{R}}(n)$, the equations have a global
bifurcation emanating from $(0,\nu_{k})$ in
\[
\{x\in X:x_{j}(t)=Rx_{j}(t+k\pi)=x_{j}(-t)\}\times\mathbb{R}^{+}\text{.}%
\]

\end{theorem}

\begin{proof}
The Fourier transform is $x=\sum_{l}x_{l}e^{ilt}$ with $x_{l}=\bar{x}_{-l}%
\in\mathbb{C}^{2n}$. Set $x_{l}=(x_{1,l},...,x_{n,l})$, where $x_{j,l}%
\in\mathbb{C}^{2}$ for $j=1,...,n$. The action of $\Gamma_{\mathbf{a}}$ in the
components $x_{j,l}\in\mathbb{C}^{2}$ is
\begin{align*}
\rho(\kappa,\pi)x_{j,l}  &  =-x_{n+1-j,l},\qquad\rho(\tilde{\kappa}%
)x_{j,l}=Rx_{j,l}\text{,}\\
\rho(\varphi)x_{j,l}  &  =e^{i\varphi l}x_{j,l},\qquad\rho(\bar{\kappa
})x_{j,l}=\bar{x}_{j,l}\text{.}%
\end{align*}
The irreducible representations under the action of $\Gamma_{\mathbf{a}}$ have
dimension one. Let $x_{j,l}=(x_{j,0,l},x_{j,1,l})\in\mathbb{C}^{2}$ with
$x_{j,k,l}\in\mathbb{C}$. The irreducible representations are given by the
subspaces $x_{j,k,l}=z$ and $x_{n+1-j,k,l}=\pm z$ for $j\in\lbrack
1,n/2)\cap\mathbb{N}$, and by $x_{j,k,l}=z$ for $j=n,n/2$. The action of the
group in coordinate $z\in\mathbb{C}$ is
\[
\rho(\tilde{\kappa})z=(-1)^{k}z,\qquad\rho(\bar{\kappa})z=\bar{z},\qquad
\rho(\pi)z=-z\text{.}%
\]

The fixed point space of $\bar{\kappa}$ consists of real $z$'s. Moreover,
every point is fixed by the action of $(\tilde{\kappa},k\pi)$.\ Therefore, the
subspace of real $z$'s is the fixed point space of $\mathbb{Z}_{2}%
(\tilde{\kappa},k\pi)\times\mathbb{Z}_{2}(\bar{\kappa})$. We define $X_{k}$ as
the intersection of $X$ and the fixed point subspace of $\mathbb{Z}_{2}%
(\tilde{\kappa},k\pi)\times\mathbb{Z}_{2}(\bar{\kappa})$ in $L_{2\pi}^{2}$,
\[
X_{k}=X\cap\mathrm{Fix}\left(  \mathbb{Z}_{2}(\tilde{\kappa},k\pi
)\times\mathbb{Z}_{2}(\bar{\kappa})\right)  \text{.}%
\]
If we prove that a positive non-resonant eigenvalue of $D^{2}V$ corresponds to
a simple zero eigenvalue of $I-T(\nu_{k})$ in the fixed point space, then the
global bifurcation follows from the global Rabinowitz alternative \cite{Ra}
applied to $Kf:X_{k}\times\mathbb{R}^{+}\rightarrow X_{k}$. A simplified proof
due to Ize is given in Theorem 3.4.1 of \cite{Ni2001}, see also the complete
exposition in \cite{IzVi03}.

The eigenvalues of $I-T(\nu)$ crossing zero are the eigenvalues of $T(\nu)$
crossing $1$. Moreover, the eigenvalues of $T(\nu)$ are $\nu^{-2}%
D^{2}V(\mathbf{a})+I$ and $\left(  \nu l\right)  ^{-2}D^{2}V(\mathbf{a})$. For
$l\neq0,1$, due to the non-resonant hypothesis of $\nu_{k}^{2}$, the matrices
$(l\nu_{k})^{-2}D^{2}V(\mathbf{a})$ have no eigenvalues equal to $1$ in
$X_{k}$. For $l=0$, the matrix $\nu_{k}^{-2}D^{2}V(\mathbf{a})+I$ has no
eigenvalues equal to $1$ because $\nu_{k}^{-2}D^{2}V(\mathbf{a})$ has no zero
eigenvalues in the fixed space of $\mathrm{Fix}\left(  \mathbb{Z}_{2}%
(\tilde{\kappa},k\pi)\right)  $. Finally, for $l=1$, the complex matrix
$\nu_{k}^{-2}D^{2}V(\mathbf{a})$ has one eigenvalue equal to $1$ corresponding
to the irreducible representation $x_{j,k,l}\in\mathbb{C}$. Therefore, in
$\mathrm{Fix}(\mathbb{Z}_{2}(\bar{\kappa}))$ (i.e. the subspace $X_{k}$), the
linear operator $I-T(\nu)$ has a simple eigenvalue crossing zero corresponding
to $l=1$.
\end{proof}

For the case $k=0$, the condition $x_{j}(t)=Rx_{j}(t+k\pi)$ implies that the
bifurcation consist of collinear periodic solutions, while if $k=1$, the
condition implies vertical orbits are degenerated figure eights.

\begin{remark}
Additionally, the action of $(\kappa,\pi)$ in the irreducible representation
$z$ is $\rho(\kappa,\pi)z=\pm z$, and every point is fixed under the action of
$(\kappa,\pi)$ or $(\kappa,\pi,\pi)$. Then, the solutions of the previous
theorem have the additional symmetry $x_{j}(t)=-x_{n+1-j}(t)$ or
$x_{j}(t)=-x_{n+1-j}(t+\pi)$.
\end{remark}

\subsection{The circular chain}

The map $Kf$ is equivariant under the action of
\[
\Gamma=D_{n}\times O(2)\times O(2).
\]
The isotropy group of the circular chain is%
\[
\Gamma_{\mathbf{a}}=\tilde{D}_{n}\times O(2).
\]
We have proved that $D^{2}V(\mathbf{a)}$ is equivalent to the matrix
$\mathrm{diag}(M_{1},...,M_{n})$, where $M_{k}=M_{n-k}\in M_{\mathbb{C}}(2)$.
Both matrices $M_{1}$ and $M_{n-1}$ have a zero-eigenvalue corresponding to
translations and $M_{n}$ has a zero-eigenvalue corresponding to rotations.

\begin{theorem}
\label{The3}Assuming that $D^{2}V(\mathbf{a})$ has only three
zero-eigenvalues, two corresponding to translations and one to rotations. If
\[
\nu_{k}=\left(  \alpha_{k}\pm\sqrt{\beta_{k}^{2}+\gamma_{k}^{2}}\right)
^{1/2}%
\]
is a non-resonant positive eigenvalue of $D^{2}V(\mathbf{a})$, double for
$k\in\lbrack1,n/2)\cap\mathbb{N}$ and simple for $k\in\{n/2,n\}\cap\mathbb{N}%
$, the equations have a global bifurcation emanating from $(0,\nu_{k})$ in%
\[
\{x\in X:x_{j}(t)=Rx_{n-j}(t)=x_{j}(-t)\}\times\mathbb{R}^{+}\text{.}%
\]

\end{theorem}

\begin{proof}
Let $x=\sum_{l}x_{l}e^{ilt}$ with $x_{l}=\bar{x}_{-l}\in\mathbb{C}^{2n}$.
Using the transformation (\ref{T}), we have
\[
x(t)=\sum_{(j,l)\in\mathbb{Z}_{n}\times\mathbb{Z}}T_{j}(x_{j,l})e^{ilt}%
\text{.}%
\]
The condition $x(t)=\bar{x}(t)$ implies $x_{j,l}=\bar{x}_{-j,-l}$. Then, the
action of $\Gamma_{\mathbf{a}}$ in the components $x_{j,l}$ is
\begin{align*}
\rho(\zeta,\zeta)x_{j,l}  &  =e^{ij\zeta}x_{j,l}\text{,\qquad}\rho
(\kappa,\tilde{\kappa})x_{j,l}=Rx_{-j,l}\text{,}\\
\rho(\varphi)x_{j,l}  &  =e^{il\varphi}x_{j,l}\text{,\qquad}\rho(\bar{\kappa
})x_{j,l}=\bar{x}_{-j,l}\text{.}%
\end{align*}
Since $x_{j,l}\in\mathbb{C}^{2}$, the irreducible representations of
$\Gamma_{\mathbf{a}}$ are
\[
(x_{j,l},x_{-j,l})=(z_{1}v_{j}^{\pm},z_{2}v_{n-j}^{\pm})\text{,}%
\]
for $j\in\lbrack1,n/2)\cap\mathbb{N}$, where $(z_{1},z_{2})\in\mathbb{C}^{2}$
and $v_{j}^{\pm}$ are the eigenvalues with the properties (\ref{eig}). The
action in $(z_{1},z_{2})$ is
\[
\rho(\kappa,\tilde{\kappa})(z_{1},z_{2})=-(z_{2},z_{1})\text{,\quad}\rho
(\bar{\kappa})(z_{1},z_{2})=(\bar{z}_{2},\bar{z}_{1})\text{.}%
\]

As in the previous theorem, we look for isotropy groups with fixed point space
of real dimension equal to one in the representation $(z_{1},z_{2}%
)\in\mathbb{C}^{2}$. A point is fixed by $\bar{\kappa}$ if $z_{2}=\bar{z}_{1}$
and by $\left(  \kappa,\tilde{\kappa}\right)  $ if $z_{1}=-\bar{z}_{1}.$ Then,
the fixed point space of $\mathbb{Z}_{2}(\kappa,\tilde{\kappa})\times
\mathbb{Z}_{2}(\bar{\kappa})$ in each irreducible representation is
$(z_{1},\bar{z}_{1})\in\mathbb{C}^{2}$ with $z_{1}$ imaginary, and has real
dimension equal to one. As in the previous theorem, the global bifurcation
follows from applying the global Rabinowitz alternative \cite{Ra} to the
operator $Kf$ restricted to $X\cap\mathrm{Fix}(\mathbb{Z}_{2}(\kappa
,\tilde{\kappa})\times\mathbb{Z}_{2}(\bar{\kappa}))$. In a similar way, the
case $k\in\{n/2,n\}\cap\mathbb{N}$ follows.
\end{proof}

\begin{remark}
Another bifurcation of periodic solutions from $(0,\nu_{k})$ may be proven
using the abelian group%
\[
\Gamma=\mathbb{Z}_{n}\times S^{1}\times S^{1}%
\]
and $\Gamma$-equivariant degree theory \cite{IzVi03}. In this way, one may
obtain an additional bifurcation of periodic solution in the fixed point space
of $\mathbb{Z}_{n}(\zeta,\zeta,-k\zeta)$. These solutions satisfy%
\[
x_{j}(t)=e^{-J\zeta}x_{j+1}(t-k\zeta)\text{.}%
\]
See \cite{GaIz13a} and \cite{GaIz13b} for details.
\end{remark}

\section{Applications}

Using CHARMM36 force field \cite{BeZh12}, we can explore the implications of
our results on actual configurations. To this end, we write Newton's equations
as
\[
m\ddot{w}_{j}=-\nabla_{w_{j}}\tilde{V}(w)\text{,}%
\]
where $\tilde{V}(w)$ is given in (\ref{V}), $\tilde{U}(x)=k(\sqrt{x}-b)^{2}$
and $\tilde{W}(x)$ is
\[
\tilde{W}(x)=4\varepsilon\left(  \frac{\sigma^{12}}{x^{6}}-\frac{\sigma^{6}%
}{x^{3}}\right)  +q\frac{1}{\sqrt{x}}.
\]

To apply our theorems, we renormalize the equations by taking $w_{j}%
(t)=bu_{j}(\omega t)$, where $\omega=\sqrt{k/m}$. Since $\nabla_{w_{j}}%
=b^{-1}\nabla_{u_{j}}$, then
\[
\ddot{u}_{j}=-\frac{1}{kb^{2}}\nabla_{u_{j}}\tilde{V}(bu)\text{.}%
\]
Therefore, $\ddot{u}_{j}=-\nabla_{u_{j}}V(u)$, where $V(u)=\frac{1}{kb^{2}%
}\tilde{V}(bu)$, and we have taken the rescaled potentials $U(x)=(\sqrt
{x}-1)^{2}$ and%
\[
W(x)=\frac{4\varepsilon}{kb^{2}}\left(  \frac{\sigma^{12}}{b^{12}x^{6}}%
-\frac{\sigma^{6}}{b^{6}x^{3}}\right)  +\frac{q}{kb^{2}}\frac{1}{b\sqrt{x}%
}=\frac{B}{x^{6}}-\frac{A}{x^{3}}+\frac{C}{\sqrt{x}},
\]
and
\[
A=\frac{4\varepsilon}{kb^{8}}\sigma^{6},\quad B=\frac{4\varepsilon}{kb^{14}%
}\sigma^{12},\quad C=\frac{q}{kb^{3}}\text{.}%
\]

For instance, carbon atoms have constants $\varepsilon\sim0.3$~$\mathrm{kJ}%
~\mathrm{mol}^{-1}$, $\sigma\sim0.35$~$\mathrm{nm}$, $b\sim0.13$~$\mathrm{nm}%
$, $k\sim255,224$~$\mathrm{kJ~}\mathrm{nm}^{-1}~\mathrm{mol}^{-2}$ and $q=0$.
Therefore, $A\sim0.1$, $B\sim40$\text{ and }$C=0$\text{. }In Table 1, we
present the $9$ non-zero eigenvalues of the Hessian for $6$ carbon atoms, and
the number of unstable eigenvalues for $n$ carbon atoms.

\begin{table}[th]%
\begin{tabular}
[c]{|c|c|c|}\hline
$n=6$ & Collinar chain & Ring\\\hline
$\lambda_{1}$ & 43.6516 & 35.0707\\
$\lambda_{2}$ & 35.089 & 29.2273\\
$\lambda_{3}$ & 23.3929 & 29.2273\\
$\lambda_{4}$ & 11.6967 & 17.5408\\
$\lambda_{5}$ & 3.13419 & 17.5408\\
$\lambda_{6}$ & -0.000225756 & 11.7003\\
$\lambda_{7}$ & -0.000129767 & 0.0109106\\
$\lambda_{8}$ & -0.0000528487 & 0.00469506\\
$\lambda_{9}$ & -0.0000104432 & 0.00469506\\\hline
\end{tabular}
\qquad\
\begin{tabular}
[c]{|c|c|c|}\hline
$n$ & Collinear chain & Ring\\\hline
3 & 1 & 0\\
4 & 2 & 0\\
5 & 3 & 0\\
6 & 4 & 0\\
7 & 5 & 0\\
8 & 6 & 0\\
9 & 7 & 2\\
10 & 8 & 2\\
11 & 9 & 4\\\hline
\end{tabular}
\caption{Left:\ Eigenvalues of the Hessian for $6$ carbon atoms.
Right:\ Number of negative eigenvalues for $n$ carbon atoms.}%
\end{table}

Other particles considered in CHARMM36 have similar parameter values, where
$k$ has a range between $1\times10^{5}$ and $5\times10^{5}$~$\mathrm{kJ~}%
\mathrm{nm}^{-1}~\mathrm{mol}^{-2}$. With these considerations, the parameters
set for real configurations is approximately given by%
\begin{equation}
(A,B)\in\lbrack0,1]\times\lbrack0,100]\text{.} \label{ran}%
\end{equation}
For convenience, we take $C=0$, as most backbone atoms are neutral in charge
and the qualitative behavior does not change considering $C=0$. We present
numerical computations for $n=6$ and this set of parameters.

\subsection{Collinear chain}

\begin{figure}[h]
\begin{center}
\resizebox{13cm}{!} {\includegraphics{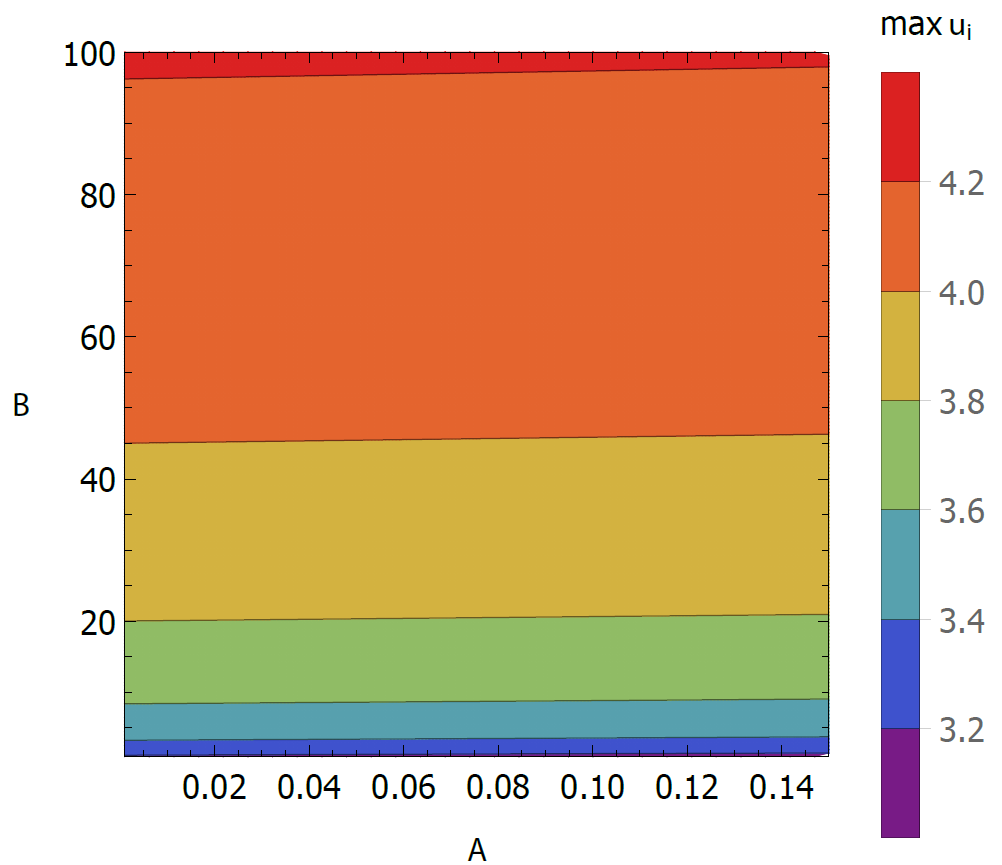}  \hskip1.0cm
\includegraphics{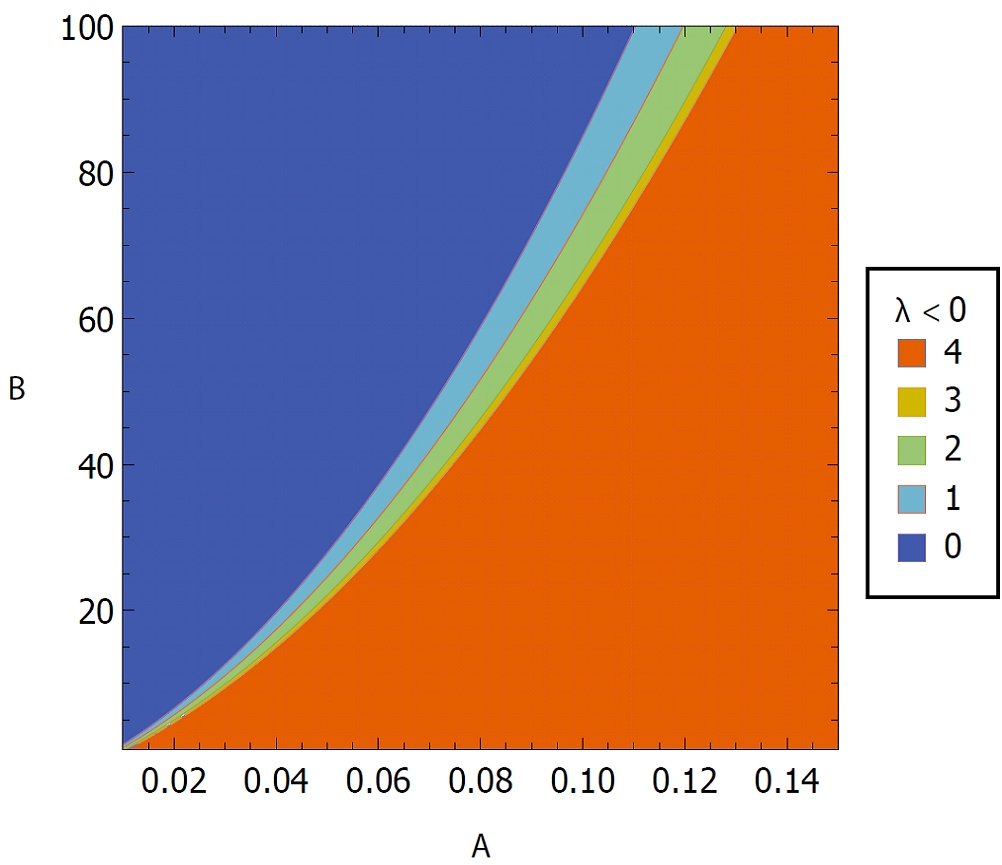} }
\end{center}
\caption{Left (2a): Half length of the collinear equilibrium for $n=6$. Right
(2b): The number of negative eigenvalues of the Hessian.}%
\label{fig:1}%
\end{figure}

Using Newton's method, we have numerically computed the global collinear
symmetric minimizer. The half length of the collinear equilibrium, which is
equal to $2.5$ when $A=B=0$, is presented in Figure 2a for $n=6$ and different parameters.

The matrix $M_{0}$ has $1$ zero-eigenvalue. We have numerically computed that
$M_{0}$ has $5$ positive eigenvalues for any $A\in\lbrack0,1]$ and
$B\in\lbrack0,100]$. Therefore, the collinear equilibrium has $5$ bifurcating
branches of collinear even periodic solutions.

Two of the six eigenvalues of $M_{1}$ are $0$. In Figure 2b we present the
number of negative eigenvalues of $M_{1}$. The collinear equilibrium is stable
in the region with zero negative eigenvalues shown in Figure 2b. In this
region, there are $4$ branches of non-collinear even periodic solutions, with
orbits resembling figure eights. The number of these solutions changes to $3$
in the region with $1$ negative eigenvalue in Figure 2b, to $2$ in the region
with $2$ negative eigenvalues, and so on.

\subsection{Circular chain}

\begin{figure}[h]
\begin{center}
\resizebox{13cm}{!} {\includegraphics{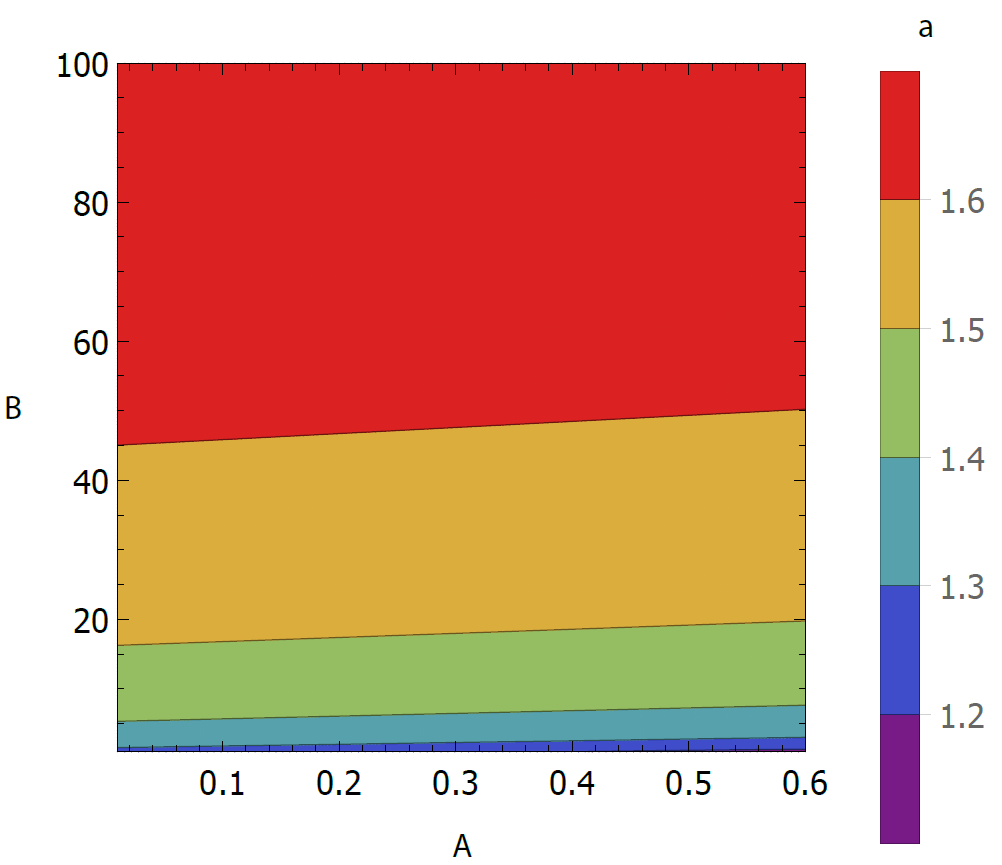}  \hskip1.0cm
\includegraphics{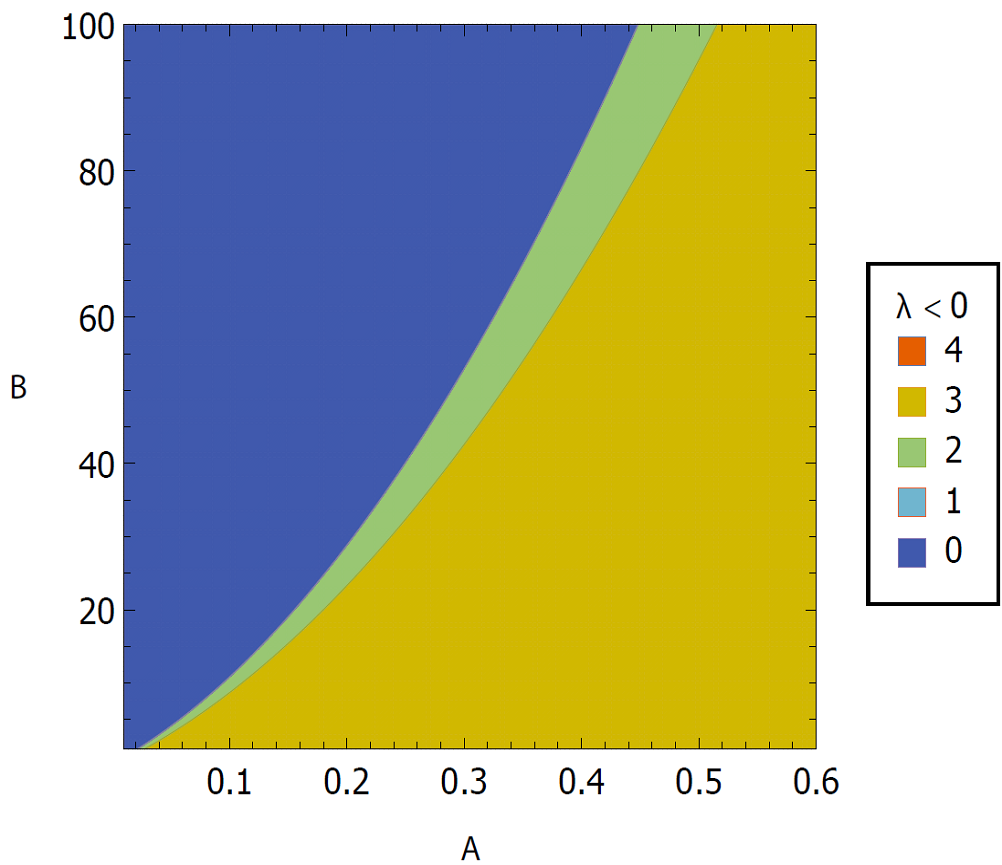} }
\end{center}
\caption{Left (3a): Amplitude of the circular equilibrium for $n=6$. Right
(3b): Number of negative eigenvalues of the Hessian counted with multiplicity
. }%
\end{figure}

For $n=6$, we present the amplitude of the global minimizer among all ring
configurations in Figure 3a, which is equal to $1$ when $A=B=0$. In this case,
the matrix $D^{2}V(a)$ has a total of $12$ eigenvalue: $3$ are zero, $3$ are
simple and $3$ are double. The circular equilibrium is stable in the region
with zero negative eigenvalues shown in Figure 3a. In this region there are
$3+3$ branches of periodic brake orbits. In addition, there are branches of
periodic solutions of the form $x_{j}(t)=e^{jJ\zeta}x_{n}(t+jk\zeta)$.

One double eigenvalue is negative in the region with two negative eigenvalues
in Figure 3b and other simple eigenvalue is negative in the region with three
negative eigenvalues. This means that the circular equilibrium is unstable in
these regions.

\section{Conclusion}

This paper is our first attempt to mathematically study stability and
vibrations of configurations of atoms. We have found collinear and circular
equilibria, and from them, we have proven the existence of periodic brake orbits.

Carbon atoms arrange in large chains known as alkanes and cycloalkanes. Our
results confirm that cycloalkanes (ring like arrangements) are stable for
$n=3$ to $8$. On the other side, we have shown that collinear arrangements are
unstable from $n=3$ to $11$; these results might explain why alkanes form
snake like structures instead.

With some work, our treatment can be extend to more complex structures such as
fullerenes. In the future, we will present a full study of other symmetric
arrangements, like the $C_{60}$ fullerene.

\section*{Acknowledgements}

C. Garc\'{\i}a is grateful to E. Perez-Chavela and S. Rybicky for useful
discussions about this problem.

\end{document}